\documentclass[11pt]{article}

\usepackage[T2A]{fontenc}
\usepackage[cp1251]{inputenc}
\usepackage{amsthm}%%%%%%%% theorem definition
\usepackage{amsfonts}
\usepackage{eufrak}
\usepackage{amssymb}
\usepackage{amsmath}
\usepackage{cite}%%%%% citations

\makeatletter

\renewcommand{\@seccntformat}[1]{\csname the#1\endcsname. }

\makeatother
\usepackage{bm}
\begin{document}
%%%%%%%%%%%%% begin theorem definition %%%%%%%%%%%%%%%%%%
\newtheoremstyle{mytheorem}
  {\topsep}   % ABOVESPACE
  {\topsep}   % BELOWSPACE
  {\itshape}  % BODYFONT
  {}       % INDENT (empty value is the same as 0pt)
  {\bfseries} % HEADFONT
  {.}         % HEADPUNCT
  {5pt plus 1pt minus 1pt} % HEADSPACE
  {}          % CUSTOM-HEAD-SPEC
\newtheoremstyle{myremark}
  {\topsep}   % ABOVESPACE
  {\topsep}   % BELOWSPACE
  {\upshape}  % BODYFONT
  {}       % INDENT (empty value is the same as 0pt)
  {\itshape} % HEADFONT
  {.}         % HEADPUNCT
  {5pt plus 0pt minus 1pt} % HEADSPACE
  {}          % CUSTOM-HEAD-SPEC\cite{}
\theoremstyle{mytheorem}
\newtheorem{theorem}{Theorem}[section]
 \newtheorem{theorema}{Theorem}
 \newtheorem*{A}{Theorem A}
  \newtheorem*{D}{Linnik theorem}
  \newtheorem*{B}{Theorem B}
 \newtheorem*{E}{O'Connor theorem}

\newtheorem{proposition}[theorem]{Proposition}
 \newtheorem{lemma}[theorem]{Lemma}
\newtheorem{corollary}[theorem]{Corollary}
\newtheorem{definition}{Definition}[section]
\theoremstyle{myremark}
\newtheorem{remark}[theorem]{Remark}
%%%%%%%%%%%%%%%%%%%%% end theorem definition %%%%%%%%%%%%%%%%%%
 \noindent{\textbf{Characterization of Probability Distributions on Locally Compact Abelian Groups}} 

\noindent{\textbf{by the Property of Identical Distribution of
 Linear Forms with Random Coefficients}}

\bigskip

\noindent\textbf{Gennadiy Feldman} https://orcid.org/0000-0001-5163-4079

\bigskip

\noindent{\textbf{Abstract}}

\bigskip

\noindent Let $X$ be a locally compact Abelian group. We consider linear
forms of independent random variables 
with values in $X$. In doing so, one of the coefficients
of the linear forms is a random variable 
with a Bernoulli distribution. For some classes of  groups 
we describe possible distributions of random variables
provided that the linear forms are identically distributed. 
The proof of the theorems is reduced to solving some functional 
equations on the character group of the group $X$,
and to solve functional equations, methods of abstract harmonic analysis are used.

\bigskip


\bigskip


\bigskip

\section{Introduction}

A large number of studies have been devoted to the problems of characterizing 
probability distributions in the case where independent random variables take 
values in a locally compact Abelian group $X$. As a rule, probability distributions are 
characterized by some properties of linear forms of independent 
random variables, provided that the 
coefficients of the linear forms are either topological automorphisms 
of $X$ or integers (see, e.g., \cite{Fe10, Fe20, {F-RIMA2025}, Fe21, M2023} and also \cite{Fe5} 
and \cite{Febooknew}, where one can find additional 
references). In this article, we consider some characterization problems on groups in 
the case where one of the coefficients of the linear forms is a random variable 
with a Bernoulli distribution.
 Although characterization of probability distributions  on the real line  
 by properties of linear forms with random coefficients
have long been studied (see \cite{LiZi, {ChP},
 {ChGe}, 
{Kl2023}}), on groups 
such a problem, as far as we know, is considered for the first time.

Let us recall that a distribution $\mu$ on the real line is a hyperbolic 
secant distribution\footnote{On the hyperbolic secant distribution, 
its properties, connections with other classical distributions and generalizations, 
see the monograph by M.J. Fischer \cite{Fi}.}
if its characteristic function is represented in the form
\begin{equation*}\label{28.1} 
\hat\mu(s)=\frac{2}{e^{\sigma s}+{e^{-\sigma s}}},\quad s \in\mathbb{R},
\end{equation*}
where $\sigma$ is a real number. It is convenient for us do not exclude the case,
when $\sigma=0$, i.e., to assume that the 
degenerate distribution concentrated at the zero is also a hyperbolic secant 
distribution.

A characterization 
theorem for the hyperbolic secant distribution was first proved by L.B.~Klebanov
in \cite{Kl2023}.
\begin{A} [{\!\!\protect\cite{Kl2023}}] Let $\xi_1$, $\xi_2$, and $\xi_3$ be independent identically distributed 
real-valued random variables with a symmetric  
distribution  $\mu$.
Let $\alpha$ be a random variable with a Bernoulli distribution taking values
$0$ and $1$ with probability $\frac{1}{2}$.
Assume that $\alpha$ is independent with 
the random vector $(\xi_1, \xi_2, \xi_3)$.
Then the following statements are equivalent:
\renewcommand{\labelenumi}{\rm(\roman{enumi})}
\begin{enumerate}
  
\item	

the linear forms 
$2\xi_1$ and \ $\xi_1+\xi_2+2\alpha\xi_3$ 
are identically distributed; 

\item
$\mu$ is a hyperbolic secant distribution.
 \end{enumerate}
\end{A}

The main part of the article (\S 2 and \S 3) is devoted to generalization of
Theorem A to the case where 
independent random variables take values in a locally compact Abelian group $X$. 
Our goal is to describe
 the possible distributions $\mu$.
The proof is reduced to describing 
the solutions of some functional equation on the  character group of 
 the group $X$. Different methods are used to find the solutions of this equation  
for the  various groups considered in the article.
In \S 4 we prove a theorem on the characterization of the 
symmetric two-point distribution
on locally compact Abelian groups containing no nonzero  compact subgroups. 
This result generalizes another theorem by Klebanov, proved in \cite{Kl2023}
for the real line.

It should be noted that many authors have studied various functional 
equations in $\mathbb{R}^n$, on groups and semigroups which either 
directly related to characterization 
problems or close to them (see, e.g.,  \cite{{A1},  {A2}, 
  {AS1},  {SaSh1},
  {Sh1},  {St3}}). 

Let us also note the following. Characterization of
 probability distributions by some properties of linear forms of independent 
random variables is reduced to solving some functional equations. 
Sometimes, these equations are solved first in the class of continuous functions, and 
then  from the obtained solutions we select those that are characteristic functions.
The functional equations arising in this article are solved directly in the class 
of characteristic functions. We do not describe continuous solutions of the 
corresponding equations.

Throughout the article we use standard results of abstract harmonic analysis.
We use the monograph by E.~Hewitt and K.A.~Ross
 \cite{HeRo1}  for references.
Let $X$ be a second countable locally compact Abelian group. We will
consider only such groups, without mentioning it specifically.
Denote by $Y$ the character
group of the group $X$, and by  $(x,y)$ the value of a character
$y \in Y$ at an element $x \in X$. 
For a closed subgroup $K$ of $X$,
denote by   $A(Y, K) = \{y \in Y: (x, y) = 1$ for all $x \in K \}$
its annihilator. 
Let $X_1$ and $X_2$ be locally compact Abelian groups with character groups 
 $Y_1$ and $Y_2$ respectively. Let $\alpha:X_1\rightarrow X_2$
be a continuous homomorphism. The adjoint homomorphism
$\widetilde\alpha: Y_2\rightarrow Y_1$
is defined by the formula  $(\alpha x_1,
y_2)=(x_1, \widetilde\alpha y_2)$ for all $x_1\in X_1$, $y_2\in
Y_2$. Put $X^{(2)}=\{2x:x\in X\}$.
 A group $X$ is said to be a 
 Corwin group if
$X^{(2)}=X$.
Denote by $\mathbb{Q}$ the  group 
of rational numbers, 
by $\mathbb{Z}$ the  group of integers,
 by
 $\mathbb{Z}(2)=\{0, 1\}$ the  group of integers 
modulo $2$, and by $\mathbb{T}$ 
the circle group.

Let $\mu$ be a probability distribution  on the group $X$.
Denote by
\begin{equation*}
\hat\mu(y) =
\int\limits_{X}(x, y)d \mu(x), \quad y\in Y,
\end{equation*}
   the characteristic function (Fourier transform) of
the distribution  $\mu$.  
Let $\nu$ be a probability
distribution  on the group $X$. The convolution
$\mu*\nu$ is defined by the formula
$$
 \mu*\nu(B)=\int\limits_{X}\mu(B-x)d \nu(x)
$$
for any Borel subset $B$ of $X$.  Let $x\in X$. Denote by $E_x$
the degenerate distribution concentrated at the element $x$.
We say that $\mu$ is symmetric if $\mu(B)=\mu(-B)$ for any Borel subset $B$ of $X$.
It is obvious that $\mu$ is symmetric if and only if $\hat\mu(y)$ 
is a real-values function.
For a compact subgroup $K$ of the group $X$ denote by $m_K$
the Haar distribution on $K$. The
characteristic function of the   distribution $m_K$ is of the
 form
\begin{equation}\label{26.13}
\widehat m_K(y)=
\begin{cases}
1  & \text{\ if\ }\   y\in A(Y, K),
\\  0  & \text{\ if\ }\ y\not\in
A(Y, K).
\end{cases}
\end{equation}

Let $f(y)$ be a function on the group $Y$. We say that $f(y)$ is a characteristic function
if there is a distribution $\mu$ on the group $X$ such that $f(y)=\hat\mu(y)$ 
for all $y\in Y$.

\section{Characterization of Probability Distributions on 
a  Locally  Compact Connected Abelian Group of Dimensional 1}

In this section we study an analogue of Theorem A in the case where 
random variables take values in a    locally  compact connected Abelian 
group $X$ of dimensional 1.  
Recall the definition of the ${\bm a}$-{adic solenoid. 
Consider a sequence  ${\bm a}=(a_0, a_1,a_2,\dots)$, where  $a_j$ are natural
and  $a_j > 1$.  Let $\Delta_{\bm a}$ be the
group of  ${\bm a}$-adic integers (\!\!\cite[(10.2)]{HeRo1}).  As a
set $\Delta_{\bm a}$ coincides
with the Cartesian product $\mathop{\mbox{\rm\bf
P}}\limits_{n=0}^\infty\{0,1,\dots ,a_n-1\}$.
Put ${u}=(1, 0,\dots,0,\dots)\in \Delta_{\bm a}$.
Denote by
 $B$ a subgroup of the group
$\mathbb{R}\times\Delta_{\bm a}$ of the form
$B=\{(n,nu)\}_{n=-\infty}^{\infty}$. The
factor-group $\Sigma_{\bm a}
=(\mathbb{R}\times\Delta_{\bm a})/B$ is called
  the ${\bm a}$-{adic solenoid}.   The
group $\Sigma_{\bm a}$ is  compact, connected, and has
dimension 1  (\!\!\cite[(10.12), (10.13),
(24.28)]{HeRo1}). The character group of the group
$\Sigma_{\bm a}$ is topologically isomorphic to a discrete additive
group $
 H_{\bm a}$ of the form
 \begin{equation*}\label{26.12}
H_{\bm a}=
\left\{\frac{m}{a_0a_1 \cdots a_n} : \ n = 0, 1,\dots; \ m
\in {\mathbb{Z}} \right\}
\end{equation*}
(\!\!\cite[(25.3)]{HeRo1}).

We need the following lemmas.

\begin{lemma}\label{l1} Let  $X$ be a locally compact Abelian group with character group $Y$.
Let $\xi_1$, $\xi_2$, and $\xi_3$ be independent identically distributed 
random variables with values in the group $X$ and   
distribution  $\mu$.
Let $\alpha$ be a random variable with a Bernoulli distribution taking values
$0$ and $1$ with probability $\frac{1}{2}$.
Assume that $\alpha$ is independent with the random 
vector $(\xi_1, \xi_2, \xi_3)$.
Then the following statements are equivalent:
\renewcommand{\labelenumi}{\rm(\roman{enumi})}
\begin{enumerate}
  
\item	

the linear forms 
$2\xi_1$ and \ $\xi_1+\xi_2+2\alpha\xi_3$ 
are identically distributed; 

\item
the characteristic function $\hat\mu(y)$ satisfies the equation
\begin{equation}\label{26.14}
\hat\mu(2y)=\hat\mu^2(y)\frac{\hat\mu(2y)+1}{2}, \quad y\in Y.
 \end{equation}
\end{enumerate}
\end{lemma}
\begin{proof}
The proof of the lemma is standard and does not differ from the 
proof for the real line. If $\xi$ is a random variable with values in
$X$ and distribution $\mu$, then $\hat\mu(y)={\bf E}[(\xi, y)]$. 
Taking into account that two random variables are identically distributed if and
only if their characteristic functions coincide, we have
\begin{multline*}
{\rm (i)}\Longleftrightarrow{\bf E}[(2\xi_1, y)]={\bf E}[(\xi_1+\xi_2+2\alpha\xi_3, y)]
\Longleftrightarrow {\bf E}[(\xi_1, 2y)]={\bf E}[(\xi_1, y)]{\bf E}[(\xi_2, y)]
{\bf E}[(2\alpha\xi_3, y)]
\\ \Longleftrightarrow {\bf E}[(\xi_1, 2y)]={\bf E}[(\xi_1, y)]{\bf E}[(\xi_2, y)]
{\bf E}[(\alpha\xi_3, 2y)]\Longleftrightarrow
\hat\mu(2y)=\hat\mu^2(y)\frac{\hat\mu(2y)+1}{2}, 
\quad y\in Y.
\end{multline*}
\end{proof}

The following statement follows from the proof of Theorem 2.1 
in \cite{Fe21}.

\begin{lemma}\label{l2}  Consider the group of dyadic 
rational numbers $\mathbb{Q}_2$, i.e.,
the subgroup of $\mathbb{Q}$ of the form
\begin{equation}\label{26.1}
\mathbb{Q}_2=\left\{\frac{m}{2^n}:   \ n = 0, 1,\dots; \ m
\in {\mathbb{Z}}\right\}.
 \end{equation}
Let $g(r)$ be a real-valued characteristic function on $\mathbb{Q}_2$ such that
$g(1)\ne 0$ and the inequalities
\begin{equation}\label{26.2}
0\le g(2r)\le g^2(r) 
\end{equation}
are fulfilled for all $|r|\le \frac{1}{2}$.
Then the function $g(r)$  is uniformly continuous on the group $\mathbb{Q}_2$ 
in the topology 
induced on $\mathbb{Q}_2$ by the topology of $\mathbb{R}$.
\end{lemma}

The proof of the following statement is standard, see, e.g., \cite[Lemma 7.2]{Fe5}.

\begin{lemma}\label{l03.1} 
Let  $X$ be a locally compact Abelian group 
with character group $Y$. Let $K$ be a compact subgroup of $X$.
Then the following statements are equivalent:
\renewcommand{\labelenumi}{\rm(\roman{enumi})}
\begin{enumerate}
  
\item	

$K$ is a Corwin group\footnote{It should be noted that Corwin groups 
first appeared in the study of group 
analogues of the famous Kac--Bernstein theorem for characterizing 
the Gaussian distribution (\!\!\cite{HR}).}; 

\item
if $2y\in A(Y, K)$ then $y\in A(Y, K)$.
\end{enumerate}
\end{lemma}

\begin{theorem}\label{th1} Let an ${\bm a}$-adic solenoid 
$\Sigma_{\bm a}$ contain  no elements of order $2$.
Let $\xi_1$, $\xi_2$, and $\xi_3$ be independent identically distributed 
random variables with values in $\Sigma_{\bm a}$ and a symmetric  
distribution  $\mu$. 
Let $\alpha$ be a random variable with a Bernoulli distribution taking values
$0$ and $1$ with probability $\frac{1}{2}$.
Assume that $\alpha$ is independent with the random vector $(\xi_1, \xi_2, \xi_3)$.
Then the following statements are equivalent:
\renewcommand{\labelenumi}{\rm(\roman{enumi})}
\begin{enumerate}
  
\item	

The linear forms $2\xi_1$ and $\xi_1+\xi_2+2\alpha\xi_3$ 
are identically distributed; 

\item
there is a continuous monomorphism 
$\tau:\mathbb{R}\rightarrow\Sigma_{\bm a}$, a hyperbolic secant distribution $M$
on $\mathbb{R}$, 
and a compact Corwin subgroup
$K$ of $\Sigma_{\bm a}$ such that $\mu=\tau(M)*m_K$.
\end{enumerate}
\end{theorem}
\begin{proof}
In order not to introduce new notation, we   identify $H_{\bm a}$
with the character group of the group $\Sigma_{\bm a}$.
Since $\Sigma_{\bm a}$ is a connected group, multiplication by 2 is an epimorphism.
Since $\Sigma_{\bm a}$ contains no elements of order
2, multiplication by 2 is a monomorphism. This implies that
multiplication by 2 is a topological automorphism of 
$\Sigma_{\bm a}$\footnote{Let ${\bm a}=(a_0, a_1,a_2,\dots)$, where  $a_j$ are natural
and  $a_j > 1$. Consider the ${\bm a}$-adic solenoid
 $\Sigma_{\bm a}$. The multiplication by 2 is a topological automorphism of 
the group $\Sigma_{\bm a}$ if and only if the set 
$\{a_0, a_1,a_2,\dots\}$ contains infinitely many even numbers.}. 
Hence  multiplication by 2 
is an automorphism of 
$H_{\bm a}$. It means that both 
$\Sigma_{\bm a}$ and $H_{\bm a}$ are
groups with unique division by 2. 
We use the proof scheme of Theorem 2.1 in \cite{Fe21}, where 
the generalized P\'olya theorem
was proved for ${\bm a}$-adic solenoids $\Sigma_{\bm a}$ provided that
there is a unique prime number $p$ such that 
$\Sigma_{\bm a}$ contains no elements of order $p$.

Assume that (i) is fulfilled. Put $f(y)=\hat\mu(y)$. By Lemma \ref{l1}, 
we have
\begin{equation}\label{26.3}
f(2y)=f^2(y)\frac{f(2y)+1}{2}, \quad y\in H_{\bm a}.
\end{equation}
Since $\mu$ is a symmetric distribution,  the function $f(y)$ is real-valued.
From equation (\ref{26.3}) it follows  that the inequalities
\begin{equation}\label{26.4}
0\le f(2y)\le f^2(y), \quad y\in H_{\bm a},
\end{equation}
are valid.
Suppose that $\mu\ne m_{\Sigma_{\bm a}}$. In view of (\ref{26.13}), there is an element
$y_0\in H_{\bm a}$, $y_0\ne 0$,   such that $f(y_0)\ne 0$.   Let $\mathbb{Q}_2$ be 
the group of dyadic rational numbers, i.e., $\mathbb{Q}_2$ is
of the form
(\ref{26.1}).
 Since $H_{\bm a}$ is a group  with unique division by $2$, 
we have $ry_0\in H_{\bm a}$ for all  $r\in \mathbb{Q}_2$. Consider
on
the group $\mathbb{Q}_2$ the function 
$$g(r)=f(ry_0), \quad r\in \mathbb{Q}_2.$$   
From
(\ref{26.4}) it follows  that  inequalities (\ref{26.2}) for 
the function $g(r)$ are fulfilled for all $r\in\mathbb{Q}_2$.
By Lemma
\ref{l2}, the function
$g(r)$ is uniformly continuous on the group $\mathbb{Q}_2$ 
in the topology induced on $\mathbb{Q}_2$ by the topology 
of  $\mathbb{R}$.
In view of (\ref{26.3}), the function  $g(r)$ satisfies the equation
 \begin{equation}\label{26.5}
g(2r)=g^2(r)\frac{g(2r)+1}{2}, \quad r\in \mathbb{Q}_2.
\end{equation}
Since the function
$g(r)$ is uniformly continuous on $\mathbb{Q}_2$ 
in the topology induced on $\mathbb{Q}_2$ by the topology 
of  $\mathbb{R}$, we extend the function $g(r)$ by continuity 
from $\mathbb{Q}_2$ to a continuous
positive definite function on $\mathbb{R}$. We keep the notation $g$
for the extended function. From (\ref{26.5}) it follows  that the extended function 
$g(s)$    satisfies
the equation
\begin{equation}\label{26.6}
g(2s)=g^2(s)\frac{g(2s)+1}{2}, \quad s\in \mathbb{R}.
\end{equation}
Inasmuch as the function $f(y)$ is real-valued, the function $g(s)$ is also real-valued.
Taking into account Lemma \ref{l1} and Theorem A,  
from  (\ref{26.6}) it follows that $g(s)$  is the  characteristic 
function of a hyperbolic 
 secant distribution on the real line. Hence we get the representation
\begin{equation*}\label{26.7}
f(y)=\frac{2}{e^{\sigma(y_0)y}+e^{-\sigma(y_0)y}},\quad y=ry_0, \ r\in \mathbb{Q}_2,
\end{equation*}
for some real $\sigma(y_0)$. It is not difficult to verify that $\sigma(y_0)$
does not depend on $y_0$. Put 
$$\sigma=\sigma(y_0)$$ and
  $$B=\{y\in H_{\bm a}: f(y)\ne 0\}.$$ Thus, we have the representation
\begin{equation}\label{26.10}
f(y)=
\begin{cases}
\frac{2}{e^{\sigma y}+e^{-\sigma y}} & \text{\ if\ }\   y\in B,
\\  0  & \text{\ if\ }\ y\not\in
B.
\end{cases}
\end{equation}
 
Let us verify that the set $B$ is a subgroup of $H_{\bm a}$. 
From equation (\ref{26.3}) it follows  
 that if   $y\in B$, then $2^{-n}y\in B$
for any natural $n$.
Take $y_1, y_2\in B$. We find from (\ref{26.10})
that   $f(2^{-n} y_j)\rightarrow 1$, $j=1, 2$,  as
$n\rightarrow\infty$. 
Let us note that the inequality
 \begin{equation}\label{26.9}
1-f(y_1+y_2)\le 2[(1-f(y_1))+(1-f(y_2))], \quad y_1, y_2\in Y,
 \end{equation}
is fulfilled for an arbitrary real-valued characteristic
function $f(y)$ on a locally compact Abelian group $Y$.
Since $f(2^{-n} y_j)\rightarrow 1$, $j=1, 2$,  as
$n\rightarrow\infty$, from (\ref{26.9}) it follows that
$2^{-n} (y_1+y_2)\in B$ for   sufficiently large  
numbers $n$.  This implies that  $y_1+y_2\in B$.

Let $M$ be the hyperbolic secant distribution 
on the group $\mathbb{R}$
with the characteristic function  
\begin{equation}\label{26.15}
\widehat M(s)=\frac{2}{e^{\sigma s}+e^{-\sigma s}}, \quad s\in \mathbb{R}.
\end{equation}
Denote by $\pi:H_{\bm a}\rightarrow \mathbb{R}$ the natural homomorphism
 $\pi(r)=r$. Consider the adjoint homomorphism  
 $\tau=\tilde \pi$, $\tau:\mathbb{R}\rightarrow \Sigma_{\bm a}$.
 Since $\pi(H_{\bm a})$ is dense in $\mathbb{R}$, the homomorphism 
$\tau$  is a monomorphism. In view of (\ref{26.16}), the characteristic function 
 of the distribution  $\tau(M)$ is of the form
\begin{equation}\label{26.16}
\widehat {\tau(M)}(y)=\frac{2}{e^{\sigma y}+e^{-\sigma y}}, 
\quad y\in H_{\bm a}.
\end{equation}
Put $K=A(\Sigma_{\bm a}, B)$. Then $K$ is a compact subgroup
of $\Sigma_{\bm a}$. Since $B=A(H_{\bm a}, K)$, from 
(\ref{26.13}) we obtain  that
\begin{equation}\label{26.17}
\widehat m_K(y)=
\begin{cases}
1  & \text{\ if\ }\   y\in B,
\\  0  & \text{\ if\ }\ y\not\in
B.
\end{cases}
\end{equation}
We find from (\ref{26.10}),  (\ref{26.16}), and (\ref{26.17}) 
that $\hat\mu(y)=\widehat{\tau(M)}(y)\widehat m_K(y)$.
Hence $\mu=\tau(M)*m_K$.
It follows from equation (\ref{26.3}) that the subgroup $B$ has the property: 
if $2y\in B$, then $y\in B$. 
By Lemma \ref{l03.1}, this implies that $K$ is a Corwin group. 
Thus, we proved (ii).

Assume that (ii) is fulfilled. Put $\pi=\tilde\tau$, 
$\pi:H_{\bm a}\rightarrow\mathbb{R}$. Then $\pi$ is of the form
$\pi(r)=ar$ for some nonzero real number $a$. Since
$M$ is a hyperbolic secant distribution,
(\ref{26.15}) is fulfilled for some real $\sigma$. 
In view of (\ref{26.13}), we have
\begin{equation}\label{12.1}
\hat\mu(y)=
\begin{cases}
\frac{2}{e^{\sigma ay}+e^{-\sigma ay}} & \text{\ if\ }\   y\in A(H_{\bm a}, K),
\\  0  & \text{\ if\ }\ y\not\in
A(H_{\bm a}, K).
\end{cases}
\end{equation}
 By Lemma \ref{l1}, 
to prove (i) it suffices to verify that the characteristic function
$f(y)=\hat\mu(y)$ satisfies  equation (\ref{26.3}). Take $y\in A(H_{\bm a}, K)$.
Then $2y\in A(H_{\bm a}, K)$ and (\ref{12.1}) implies that (\ref{26.3}) is fulfilled.
Take $y\notin A(H_{\bm a}, K)$. Since $K$ is a Corwin group, it follows from 
 Lemma \ref{l03.1} that 
 $2y\notin A(H_{\bm a}, K)$. In this case both sides of (\ref{26.3}) are equal to 
the zero. The theorem is completely proved. 

Note that in the proof that (ii)$\Longrightarrow$(i) we did not use the fact
that the ${\bm a}$-adic solenoid 
$\Sigma_{\bm a}$ contains  no elements of order $2$.\end{proof}

Taking into account (\ref{26.13}), Theorem \ref{th1} implies the following statement.

\begin{corollary}\label{c1} Let an ${\bm a}$-adic solenoid 
$\Sigma_{\bm a}$ contain  no elements of order $2$.
Let $\xi_1$, $\xi_2$, and $\xi_3$ be independent identically distributed 
random variables with values in $\Sigma_{\bm a}$ and a symmetric distribution  $\mu$ 
with nonvanishing 
characteristic function. Let $\alpha$ be a random variable with a Bernoulli 
distribution taking values
$0$ and $1$ with probability $\frac{1}{2}$.
Assume that $\alpha$ is independent with the random vector $(\xi_1, \xi_2, \xi_3)$.
Then the following statements are equivalent:
\renewcommand{\labelenumi}{\rm(\roman{enumi})}
\begin{enumerate}
  
\item	

the linear forms $2\xi_1$ and $\xi_1+\xi_2+2\alpha\xi_3$ 
are identically distributed; 

\item
there is a continuous monomorphism 
$\tau:\mathbb{R}\rightarrow\Sigma_{\bm a}$ and a hyperbolic secant distribution $M$
on $\mathbb{R}$ such that $\mu=\tau(M)$.
\end{enumerate}
\end{corollary}

Note also that the following statement is true. 
\begin{proposition}\label{c2} Let  $X$ be a locally compact Abelian group and 
let $K$ be a compact subgroup of $X$.
Let $\xi_1$, $\xi_2$, and $\xi_3$ be independent identically distributed 
random variables with values in $X$ and    
distribution  $m_K$. 
Let $\alpha$ be a random variable with a Bernoulli distribution taking values
$0$ and $1$ with probability $\frac{1}{2}$.
Assume that $\alpha$ is independent with the random vector $(\xi_1, \xi_2, \xi_3)$.
Then the following statements are equivalent:
\renewcommand{\labelenumi}{\rm(\roman{enumi})}
\begin{enumerate}
  
\item	

the linear forms $2\xi_1$ and $\xi_1+\xi_2+2\alpha\xi_3$ 
are identically distributed; 

\item
$K$ is a Corwin group.
\end{enumerate}
\end{proposition}

\begin{proof}
Denote by $Y$ the character group of the group $X$. 
Put $f(y)=\widehat m_K(y)$.
Assume that (i) is fulfilled. By Lemma \ref{l1}, 
the function $f(y)$ satisfies equation (\ref{26.3}). Assume that $2y\in A(Y, K)$.
In view of (\ref{26.13}), it follows from (\ref{26.3}) that $y\in A(Y, K)$.
By Lemma \ref{l03.1}, $K$ is a Corwin group. 

Assume that (ii) is fulfilled. Consider equation (\ref{26.3}). If 
$y\in A(Y, K)$, then both sides of equation  (\ref{26.3}) are equal to 1.
If $y\notin A(Y, K)$, then by Lemma \ref{l03.1}, $2y\notin A(Y, K)$ and
both sides of equation  (\ref{26.3}) are equal to 0. We see that 
the characteristic function $\widehat m_K(y)$ satisfies equation (\ref{26.3}).
By Lemma \ref{l1}, 
the linear forms $2\xi_1$ and $\xi_1+\xi_2+2\alpha\xi_3$ 
are identically distributed.
\end{proof}

Theorem \ref{th1} fails if 
an ${\bm a}$-adic solenoid 
$\Sigma_{\bm a}$ contains  an element of order $2$. Namely, the following 
statement is true.

\begin{proposition}\label{p1}
Assume that an ${\bm a}$-adic solenoid 
$\Sigma_{\bm a}$ contains  an element of order $2$. Let $\alpha$ be a random 
variable with a Bernoulli 
distribution taking values
$0$ and $1$ with probability $\frac{1}{2}$. Then there are 
independent identically distributed 
random variables $\xi_1$, $\xi_2$, and $\xi_3$ with values in $\Sigma_{\bm a}$ 
and a symmetric distribution  $\mu$ with nonvanishing 
characteristic function such that $\alpha$ is independent with the random vector $(\xi_1, \xi_2, \xi_3)$, the linear forms 
$2\xi_1$ and \ $\xi_1+\xi_2+2\alpha\xi_3$
are identically distributed, whereas
$\mu$ is not represented in the form $\mu=\tau(M)$, where $\tau$ is 
a continuous monomorphism 
$\tau:\mathbb{R}\rightarrow\Sigma_{\bm a}$ and $M$ is a hyperbolic secant distribution  
on $\mathbb{R}$.
\end{proposition}

\begin{proof}
Let $H_{\bm a}$ be the character group of the group $\Sigma_{\bm a}$. 
Since the group $\Sigma_{\bm a}$ contains  an element of order $2$, the 
sequence  ${\bm a}=(a_0, a_1,a_2,\dots)$ contains only finite number
even $a_j$. Without loss of generality we can suppose that all $a_j$ are odd.
Consider the set
\begin{equation*}\label{07.1}
H=\left\{z=\frac{m}{n}\in H_{\bm a}:m, n \ \mbox{are odd}\right\}.
\end{equation*}
Take $z\in H$ and put $L_z=\{2^kz\}_{k=0}^{\infty}$. Then the sets $L_z$
do not intersect and
$$
H_{\bm a}=\{0\}\cup\bigcup_{z\in H}L_z.  
$$

Let us define a function $c(y)$ on the set $L_z$ as follows. Take
  a real number $c(z)$ such that $0<c(z)<1$. Suppose that 
 the value $c(2^kz)$ of the function $c(y)$ is already defined. Put
 \begin{equation}\label{26.18}
c(2^{k+1}z)=\frac{c^2(2^kz)}{2-c^2(2^kz)}.  
\end{equation}
It is obvious that for any $y\in H_{\bm a}$, $y\ne 0$,  there is a unique element
$z\in H$ and a nonnegative integer $k$ such that $y\in L_z$ and $y=2^kz$. 
We define the function $f(y)$ on the 
group $H_{\bm a}$ as follows:
$$
f(y)=
\begin{cases}
c(2^kz) & \text{\ if\ }\   y=2^kz\in L_z,
\\  1  & \text{\ if\ }\ y=0. 
\end{cases}
$$
In view of (\ref{26.18}) the function $f(y)$ satisfies equation (\ref{26.3}).
We can choose the numbers $c(z)$ in such a way that
$c(-z)=c(z)$ for all $z\in H$ and
$$
\sum_{z\in H}c(z)<\frac{1}{2}.
$$
It follows from (\ref{26.18}) that $c(2^{k+1}z)<c^2(2^kz)$, $k=0, 1, \dots$, and 
hence $c(2^{k}z)<c^{2^k}(z)$, $k=1, 2, \dots$.  This implies that
$$
\sum_{y\in L_z}c(y)<2c(z),
$$
and hence
\begin{equation}\label{26.19}
\sum_{y\in H_{\bm a}, \ y\ne 0}f(y)<1.  
\end{equation}
Consider on the group $\Sigma_{\bm a}$ the function 
\begin{equation}\label{n26.19}
\rho(g)=1+\sum_{y\in H_{\bm a}, \ y\ne 0}f(y)\overline{(g, y)}, \quad g\in\Sigma_{\bm a}.
\end{equation}
It follows from (\ref{26.19}) that $\rho(g)>0$ for all $g\in\Sigma_{\bm a}$. 
It is obvious that
$$
\int\limits_{\Sigma_{\bm a}}\rho(g)dm_{\Sigma_{\bm a}}(g)=1.
$$
Denote by $\mu$ the distribution on the group $\Sigma_{\bm a}$ with the density $\rho(g)$ with respect to $m_{\Sigma_{\bm a}}$. In view of (\ref{n26.19}), we have
 $\hat\mu(y)=f(y)$ for all $y\in H_{\bm a}$. Since the function $f(y)$
 is real-valued, the distribution $\mu$ is symmetric.
Let $\xi_1$, $\xi_2$, and $\xi_3$ be independent identically distributed 
random variables with values in $\Sigma_{\bm a}$ and distribution  $\mu$ 
such that $\alpha$ is independent with $\xi_1$, $\xi_2$, 
and $\xi_3$. 
Since the function $f(y)$ satisfies equation (\ref{26.3}), by Lemma \ref{l2},
the linear forms $2\xi_1$ and $\xi_1+\xi_2+2\alpha\xi_3$ 
are identically distributed. Taking into account
(\ref{26.16}), we can choose the numbers $c(z)$ in such a way
that the distribution $\mu$ is not represented in the 
form $\mu=\tau(M)$, where $\tau$ is 
a continuous monomorphism 
$\tau:\mathbb{R}\rightarrow\Sigma_{\bm a}$ and $M$ is a 
hyperbolic secant distribution  
on $\mathbb{R}$.  Moreover, the above reasoning shows that we can hardly 
expect to find a description of possible distributions $\mu$.
\end{proof}

\begin{remark}\label{r27.1}
Consider now the case where random variables $\xi_1$, $\xi_2$, and $\xi_3$ 
take values in the circle group $\mathbb{T}$.
 
 Let $\alpha$ be a random 
variable with a Bernoulli 
distribution taking values
$0$ and $1$ with probability $\frac{1}{2}$. It is easy to see that then there are 
independent identically distributed 
random variables $\xi_1$, $\xi_2$, and $\xi_3$ with values in 
the circle group $\mathbb{T}$ 
and a symmetric distribution  $\mu$ with nonvanishing 
characteristic function such that $\alpha$ is independent 
with the random vector $(\xi_1, \xi_2, \xi_3)$, the linear forms 
$2\xi_1$ and \ $\xi_1+\xi_2+2\alpha\xi_3$
are identically distributed, whereas
$\mu$ is not represented in the form $\mu=\tau(M)$, where $\tau$ is 
a continuous homomorphism 
$\tau:\mathbb{R}\rightarrow\mathbb{T}$ and $M$ is a hyperbolic secant distribution  
on $\mathbb{R}$.
\end{remark}
 
Let $X$ be a  locally  compact connected Abelian group of dimensional 1. Then
$X$ is topologically isomorphic either to the group of real numbers $\mathbb{R}$, 
or an ${\bm a}$-adic solenoid $\Sigma_{\bm a}$, or  the circle group $\mathbb{T}$.
Let $\xi_1$, $\xi_2$, and $\xi_3$ be independent identically distributed 
random variables with values in $X$ and a symmetric  
distribution  $\mu$. 
Let $\alpha$ be a random variable with a Bernoulli distribution taking values
$0$ and $1$ with probability $\frac{1}{2}$. 
Assume that $\alpha$ is independent with 
the random vector $(\xi_1, \xi_2, \xi_3)$.
Theorems A and \ref{th1},  
Proposition  \ref{p1} and Remark \ref{r27.1} answer the question on 
the possibility of describing distributions
$\mu$ provided that the linear forms $2\xi_1$ and $\xi_1+\xi_2+2\alpha\xi_3$ 
are identically distributed.

\section{Characterization of Probability Distributions on Locally Compact 
Totally Disconnected Abelian Groups}

In this section we prove an analogue of Theorem A in the case where 
random variables take values in a locally compact
 totally disconnected Abelian group $X$. Taking into account the structure
 of the group $X$, quite unexpected is the fact that the proof uses Theorem A.
For the proof we need two lemmas. The first lemma is well known. 
For the proof see, e.g., \cite[Proposition 2.10]{Febooknew}.

\begin{lemma}\label{l3}  Let $X$ be a locally compact Abelian group  
with character group
 $Y$
and let $\mu$ be a distribution on $X$. Then the set  
$E=\{y\in Y:  \hat\mu(y)=1\}$ is a closed subgroup  of the group $Y$ and 
the distribution   $\mu$ is supported in  $A(X,E)$. Moreover, 
if the minimal subgroup
containing the support of $\mu$ coincides with $X$, then
\begin{equation}\label{1.1}
\{y\in Y:  \hat\mu(y)=1\}=\{0\}.
\end{equation}
\end{lemma}

Let $p$ be a prime number. An Abelian group  $X$ is called  $p$-{group} 
if the order of every 
element of $X$ is a power of $p$. 

\begin{lemma}\label{l4}  Let $H$ be a discrete $2$-group. 
Let $f(y)$ be a characteristic function on $H$ satisfying equation
$\rm (\ref{26.3})$. Assume that  
\begin{equation}\label{1.2}
\{y\in H:  f(y)=1\}=\{0\}.
\end{equation}
Then $H$ is isomorphic to the  group 
 $\mathbb{Z}(2)$ and $f(1)=-1$.
\end{lemma}
\begin{proof}
Let $y_1\in H$ and let $y_1$ be an element of order 2. Substituting $y=y_1$ into equation 
(\ref{26.3}), we get $f^2(y_1)=1$. In view of (\ref{1.2}), $f(y_1)=-1$. 
Let $y_2\in H$ and let $y_2$ be an element of order 4. Substituting $y=y_2$ into 
equation (\ref{26.3}) and taking into account that $f(2y_2)=-1$, we obtain a contradiction. 
Hence the group $H$ contains no elements of order 4. Assume that there is an element
$y_3\in H$ of order 2, such that $y_1\ne y_3$. Consider the subgroup $L$
of $H$ generated by the elements $y_1$ and $y_3$. 
Then $L$ is isomorphic to  $\mathbb{Z}(2)\times \mathbb{Z}(2)$. On the one hand, 
as noted above, 
\begin{equation}\label{1.4}
f(y)=-1 \ \mbox{for all} \ y\in L, \ y\ne 0.
\end{equation}
 On the other hand, 
from (\ref{1.4}) it follows  that $|f(y)|=1$ for all $y\in L$. Since $f(y)$ 
is a characteristic function, this implies
that the function $f(y)$ is a character of the group $L$, that contradicts (\ref{1.4}).
Hence the group $H$ contains the only one element  of order 2. The lemma is proved.
\end{proof}

\begin{theorem}\label{th2} Let $X$ be a locally compact 
totally disconnected Abelian group.
Let $\xi_1$, $\xi_2$, and $\xi_3$ be independent identically distributed 
random variables with values in $X$ and a symmetric  
distribution  $\mu$. 
Let $\alpha$ be a random variable with a Bernoulli distribution taking values
$0$ and $1$ with probability $\frac{1}{2}$.
Assume that $\alpha$ is independent with the random vector $(\xi_1, \xi_2, \xi_3)$.
Then the following statements are equivalent:
\renewcommand{\labelenumi}{\rm(\roman{enumi})}
\begin{enumerate}
  
\item	

the linear forms $2\xi_1$ and $\xi_1+\xi_2+2\alpha\xi_3$ 
are identically distributed; 

\item
there is a compact Corwin subgroup $K$ of the group $X$ such that either $\mu=m_K$
or $\mu=m_K*E_x$, where $x$ is an element of order $2$.
\end{enumerate}
\end{theorem}
\begin{proof}
By the structure theorem, any locally compact Abelian group 
is topologically isomorphic to a group  of the form 
$\mathbb{R}^n\times G$, where $n\ge 0$ and $G$ a locally compact Abelian group
containing a compact open subgroup (\!\!\cite[(24.30)]{HeRo1}).
Denote by $Y$ the character group of the group $X$ and by $c_Y$ the connected 
component of the zero of the group $Y$. Taking into account that any 
locally compact connected Abelian group is topologically isomorphic to 
a group  of the form 
$\mathbb{R}^m\times K$, where $m\ge 0$ and $K$ is a compact connected
Abelian group, we conclude that $c_Y$ is a compact connected Abelian group.

Let us prove (i) $\Longrightarrow$ (ii). In view of  Lemma \ref{l3}, we can suppose
that (\ref{1.1}) is fulfilled.
 Since $c_Y$ is a compact connected Abelian group, there is a 
 continuous homomorphism
$\tau:\mathbb{R}\rightarrow c_Y$ such that $\tau(\mathbb{R})$ is 
dense in $c_Y$ (\!\!\cite[(9.2), (25.18)]{HeRo1}).
Consider on the real line the function
$$
f(s)=\hat\mu(\tau(s)), \quad s\in \mathbb{R}.
$$
Then $f(s)$ is a characteristic function on the real line.
By Lemma \ref{l1}, the characteristic function $\hat\mu(y)$ satisfies 
equation (\ref{26.14}).
Hence the function $f(s)$ satisfies the equation 
\begin{equation}\label{1.5} 
f(2s)=f^2(s)\frac{f(2s)+1}{2}, \quad s\in \mathbb{R}.
\end{equation}
Since $\mu$ is a symmetric distribution,  the function $f(s)$ is real-valued.
Applying Lemma \ref{l1} and Theorem A,  from (\ref{1.5}) we obtain 
\begin{equation}\label{1.3}  
f(s)=\frac{2}{e^{\sigma s}+{e^{-\sigma s}}}, \quad s\in \mathbb{R},
\end{equation} 
where $\sigma$ is a real number. Since $\hat\mu(0)=1$, we can take a neighborhood 
$U$ of the zero of the group $c_Y$ and $a>0$
in such a way that $\hat\mu(y)\ge a$ for all $y\in U$. 
As much as $\tau(\mathbb{R})$ is dense in $c_Y$, it is easy to see that there is
a sequence $s_n\rightarrow\infty$ such that $\tau(s_n)\in U$ for all $s_n$.
This implies that $f(s_n)\ge a$ for all $s_n$.
If $\sigma\ne 0$ in (\ref{1.3}), then $f(s_n)\rightarrow 0$, as 
$s_n\rightarrow\infty$.
The obtained contradiction shows that $\sigma=0$. We got that $\hat\mu(y)=1$ for
all $y\in\tau(\mathbb{R})$. Since $\tau(\mathbb{R})$ is dense in $c_Y$,  
we obtained that $\hat\mu(y)=1$ for
all $y\in c_Y$. Inasmuch as  (\ref{1.1}) is fulfilled, $c_Y=\{0\}$, i.e., 
the group $Y$ is totally 
disconnected. 

The fact that $Y$ is a totally 
disconnected group implies that any neighborhood of the zero of the group $Y$
contains a compact open subgroup (\!\!\cite[(7.7)]{HeRo1}). Let us denote this 
subgroup by $L$. We can suppose that
$$
\min_{y\in L}\hat\mu(y)=\hat\mu(y_0)>0,
$$
where $y_0\in L$. Substituting $y=y_0$ into equation (\ref{26.14}), we get 
$\hat\mu(2y_0)\le\hat\mu^2(y_0)<\hat\mu(y_0)$ if $\hat\mu(y_0)<1$.
Since $2y_0\in L$, we conclude that $\hat\mu(y_0)=1$, and hence 
$\hat\mu(y)=1$ for all $y\in L$. As much as we assume that (\ref{1.1}) is fulfilled,
$L=\{0\}$. Since $L$ is an open subgroup, 
 $Y$ is a discrete group. This implies that $X$ is a compact group.
  Thus, we have reduced the proof of the theorem to the case when
$X$ is a compact totally disconnected Abelian group. This implies that
$Y$ is a discrete torsion group (\!\!\cite[(24.26)]{HeRo1}). 
Consider two cases.

1. The group $Y$ contains no elements of order 2. 
 Since $Y$ is a discrete torsion  group,
multiplication by 2 is an automorphism of $Y$.  It follows from this
that  for any $y\in Y$ there is
a natural $m$ such that $2^my=y$. Generally speaking, $m$ depends on $y$.
Moreover, taking 
into account that
$Y^{(2)}=Y$ and the characteristic function $\hat\mu(y)$ satisfies 
equation (\ref{26.14}), we find from  
(\ref{26.14}) that $\hat\mu(y)\ge 0$ for all $y\in Y$ and 
$\hat\mu(2y)\le \hat\mu(y)$.   Taking into account the above, we get
$$
\hat\mu(y)=\hat\mu(2^my)\le\hat\mu(2^{m-1}y)\le\cdots
\le\hat\mu(2y)\le\hat\mu(y), \quad y\in Y.
$$
This implies that $\hat\mu(2y)=\hat\mu(y)$ for all $y\in Y$. 
Assume that $\hat\mu(y_0)\ne 0$ for some $y_0\in Y$. Substituting $y=y_0$ into equation  
(\ref{26.14}), we find that $\hat\mu(y_0)=1$, and
in view of (\ref{1.1}), $y_0=0$. Thus, we obtain the following representation
\begin{equation*} 
\hat\mu(y)=
\begin{cases}
1  & \text{\ if\ }\   y=0,
\\  0  & \text{\ if\ }\ y\ne 0.
\end{cases}
\end{equation*}
It follows from (\ref{26.13}) that $\mu=m_X$.  Since multiplication by 2 is an 
automorphism of the group $Y$,  multiplication by 2 is a topological 
automorphism of the group $X$. Hence  $X$ is a Corwin group.  
 
2. Assume that the group $Y$ contains an element of order 2. 
Any torsion Abelian group is isomorphic to the weak 
direct product of its $p$-component, where $p$ is a prime number,
and the $p$-component
is the subgroup consisting of all elements whose order is a power of
$p$ (\!\!\cite[(A.3)]{HeRo1}). 
It follows from Lemma \ref{l4} and (\ref{1.1}) that 2-component of the group $Y$ 
is isomorphic to $\mathbb{Z}(2)$. 
This implies that the group $Y$ is isomorphic
to a group of the form $\mathbb{Z}(2)\times L$, where $L$ is a discrete torsion
group containing no elements of order 2.
Then the group $X$ is topologically isomorphic
to a group of the form $\mathbb{Z}(2)\times K$, where $K$ is a compact 
Corwin group.  To avoid introducing new notation,
suppose that $X=\mathbb{Z}(2)\times K$ and $Y=\mathbb{Z}(2)\times L$.
Denote by $(g, k)$, where $g\in\mathbb{Z}(2)$, $k\in K$, elements of
the group $X$ and by $(h, l)$, where $h\in\mathbb{Z}(2)$, $l\in L$, elements of
the group $Y$.  As was proved in item 1, $\hat\mu(0, l)=0$, for all 
$l\in L$, $l\ne 0$. Substituting $y=(1, l)$, $l\ne 0$, into equation (\ref{26.14}) 
and taking into account that $\hat\mu(0, 2l)=0$, we find that 
$\hat\mu(1, l)=0$.
Moreover, by Lemma \ref{l4}, $\hat\mu(1, 0)=-1$. Thus, we get the following 
representation
\begin{equation}\label{1.6}
\hat\mu(h, l)=
\begin{cases}
1  & \text{\ if\ }\   h=0, \ l=0,
\\  -1  & \text{\ if\ }\ h=1, \ l=0,
\\ 0  & \text{\ if\ }\ l\ne 0.
\end{cases}
\end{equation}
Taking into account that $A(Y, K)=\mathbb{Z}(2)$ and (\ref{26.13}), we have
\begin{equation}\label{3.1}
\widehat m_K(h, l)=
\begin{cases}
1  & \text{\ if\ }\   l=0,
\\  0  & \text{\ if\ }\ l\ne 0.
\end{cases}
\end{equation}
Also note that
\begin{equation}\label{3.2}
\hat E_{(1, 0)}(h,l)=((1, 0), (h,l))=\begin{cases}
1  & \text{\ if\ }\   h=0,
\\  -1  & \text{\ if\ }\ h=1.
\end{cases}
\end{equation}
It follows from (\ref{1.6})--(\ref{3.2}) that 
$\hat\mu(h,l)=\widehat m_K(h,l)\hat E_{(1, 0)}(h,l)$.
Hence $\mu=m_K*E_{(1, 0)}$. Thus, we proved that (i) $\Longrightarrow$ (ii).

Let us prove (ii) $\Longrightarrow$ (i). Assume that 
$K$ is a compact Corwin group. 
Let $\mu=m_K$.   
By Proposition \ref{c2}, the linear forms $2\xi_1$ and $\xi_1+\xi_2+2\alpha\xi_3$ 
are identically distributed.

Let  $\mu=m_K*E_x$, where $x$ is an element of order $2$.
By Lemma \ref{l1}, 
to prove (i) it suffices to verify that the characteristic function
$\hat\mu(y)$ satisfies  equation (\ref{26.14}). 
 We have
$\hat\mu(y)=\widehat m_K(y)(x, y)$.
If $y\in A(Y, K)$, then $2y\in A(Y, K)$ and  
$\widehat m_K(y)=\widehat m_K(2y)=1$, $(x, y)=\pm 1$,
$(x, 2y)=1$. This implies that both sides of equation (\ref{26.14}) are equal to 1.
If $y\notin A(Y, K)$, by Lemma \ref{l03.1}, this implies that
 $2y\notin A(Y, K)$ and both sides of equation
(\ref{26.14}) are equal to 0. 
Hence  the characteristic function
$\hat\mu(y)$ satisfies  equation (\ref{26.14}). 
The theorem is completely proved.
\end{proof}

\section{Characterization of the Symmetric Two-point Distribution on 
Locally Compact Abelian Groups Containing no Nonzero Compact Subgroups}

In \cite{Kl2023}, L.B.~Klebanov proved the following characterization 
theorem for the symmetric two-point distribution on the real line.
\begin{B} [{\!\!\protect\cite{Kl2023}}] 
Let $\xi_1$  and $\xi_2$ be independent identically distributed 
real-valued random variables with a symmetric  
distribution  $\mu$. 
Let $\alpha$ be a random variable with a Bernoulli distribution taking values
$0$ and $1$ with probability $\frac{1}{2}$.
Assume that $\alpha$ is independent with the random vector $(\xi_1, \xi_2)$.
Then the following statements are equivalent:
\renewcommand{\labelenumi}{\rm(\roman{enumi})}
\begin{enumerate}
  
\item	

the linear forms $2\alpha\xi_1$ and $\xi_1+\xi_2$ 
are identically distributed; 

\item
 there is a real number $a$ such that $\mu=\frac{1}{2}\left(E_{a}+E_{-a}\right)$.
\end{enumerate}
\end{B}

In this section we prove an analogue of this theorem for  
locally compact Abelian groups containing no nonzero compact subgroups.
The proof is based on the following statement proved by 
T.A.~O'Connor in \cite{O}.

\begin{E} Let $Y$ be a locally compact connected Abelian group. 
Let $h(y)$ be a continuous bounded real-valued function on $Y$
satisfying d'Alembert's functional equation\footnote{d'Alembert's 
functional equation
and its generalizations on various algebraic structures 
have been studied by many authors. A comprehensive overview of the 
corresponding results 
can be found in Chapters 8--10 of the monograph  by H.~Stetkaer \cite{St3}.} 
\begin{equation}\label{18.1}
 h(u+v)+h(u-v)=2h(u)h(v), \quad u, v\in Y,
 \end{equation}
and the condition $h(0)=1$. Then there is $x_0\in X$, where $X$ is 
the character group of the group $Y$, such that $h(y)={\rm Re}\ (x_0, y)$, 
$y\in Y$.
\end{E}

\begin{lemma}\label{nl1} Let  $X$ be a locally compact Abelian group 
with character group $Y$.
Let $\xi_1$  and $\xi_2$ be independent identically distributed 
random variables with values in $X$ and distribution  $\mu$. 
Let $\alpha$ be a random variable with a Bernoulli distribution taking values
$0$ and $1$ with probability $\frac{1}{2}$.
Assume that $\alpha$ is independent with the random vector $(\xi_1, \xi_2)$.
Then the following statements are equivalent:
\renewcommand{\labelenumi}{\rm(\roman{enumi})}
\begin{enumerate}
  
\item	

the linear forms 
$2\alpha\xi_1$ and \ $\xi_1+\xi_2$ 
are identically distributed; 

\item
the characteristic function $\hat\mu(y)$ satisfies the equation
\begin{equation}\label{n26.14}
 \hat\mu(2y)=2\hat\mu^2(y)-1, \quad y\in Y.
 \end{equation}
\end{enumerate}
\end{lemma}
We omit the proof of the lemma, since it is analogous to the proof of Lemma \ref{l1}. 

\begin{theorem}\label{nth1} Let $X$ be a locally compact Abelian group
 containing no nonzero  compact subgroups.
Let $\xi_1$  and $\xi_2$ be independent identically distributed 
random variables with values in $X$ and a symmetric  
distribution  $\mu$. 
Let $\alpha$ be a random variable with a Bernoulli distribution taking values
$0$ and $1$ with probability $\frac{1}{2}$.
Assume that $\alpha$ is independent with the random vector $(\xi_1, \xi_2)$.
Then the following statements are equivalent:
\renewcommand{\labelenumi}{\rm(\roman{enumi})}
\begin{enumerate}
  
\item	

the linear forms $2\alpha\xi_1$ and $\xi_1+\xi_2$ 
are identically distributed; 

\item
 there is $x_0\in X$ such that $\mu=\frac{1}{2}\left(E_{x_0}+E_{-x_0}\right)$.
\end{enumerate}
\end{theorem}
\begin{proof} By the structure theorem for locally compact Abelian groups,
each such group is topologically isomorphic to a group of the form
$\mathbb{R}^n\times G$, where $n\ge 0$ and a locally compact Abelian group
$G$ contains a compact open subgroup (\!\!\cite[(24.30)]{HeRo1}).
Since the group $X$  contains no nonzero  compact subgroups, this compact open 
subgroup is the zero, i.e., $G$ 
is a discrete group. Moreover, $G$ is a torsion free group.  Denote by $Y$ the character 
group of the group $X$. Then $Y$ is topologically isomorphic to a group of the form
$\mathbb{R}^n\times H$, where $H$ is the character group of the group $G$.
To avoid introducing additional notation, we assume that $X=\mathbb{R}^n\times G$ and
$Y=\mathbb{R}^n\times H$.  

Let us prove (i)$\Longrightarrow$(ii). Inasmuch as the linear forms $2\alpha\xi_1$ 
and $\xi_1+\xi_2$ 
are identically distributed,
by Lemma \ref{nl1}, the characteristic function $\hat\mu(y)$
satisfies equation (\ref{n26.14}). 

1. First suppose that $X=G$. Then $Y=H$. Since 
$G$ is a  discrete torsion free Abelian group, $H$ is a compact connected Abelian group 
(\!\!\cite[(23.17), (24.25)]{HeRo1}). Hence
there is a continuous homomorphism
$\tau: \mathbb{R}\rightarrow H$ such that $\tau(\mathbb{R})$ 
is dense in $H$ (\!\!\cite[(9.2), (25.18)]{HeRo1}).
 Consider the restriction of equation (\ref{n26.14}) to the subgroup 
 $\tau(\mathbb{R})$, i.e., assume that $y=\tau(s)$, $s\in \mathbb{R}$,
  in (\ref{n26.14}).
Put $f(s)=\hat\mu(\tau(s))$, 
 $s\in \mathbb{R}$.
 Then $f(s)$ is a characteristic function on $\mathbb{R}$
 satisfying the equation 
\begin{equation}\label{1n26.14}
 f(2s)=2f^2(s)-1, \quad s\in \mathbb{R}.
 \end{equation}
Since $\mu$ is a symmetric distribution,  the function $f(s)$ is real-valued.
It follows from Lemma \ref{nl1} and Theorem B that $f(s)=\cos as$ for 
some real number $a$. 
Hence the function $f(s)$ satisfies  d'Alembert's functional equation
on $\mathbb{R}$. This implies
that the function $\hat\mu(y)$ satisfies  d'Alembert's functional equation 
on the subgroup $\tau(\mathbb{R})$.
 Since the subgroup $\tau(\mathbb{R})$ is dense in $H$, the function $\hat\mu(y)$ 
 satisfies  d'Alembert's functional equation on the group $H$. By O'Connor 
 theorem,  applying to the function 
 $h(y)=\hat\mu(y)$, $y\in H$, there is $x_0\in X$ such that 
 $\hat\mu(y)={\rm Re}\ (x_0, y)$, $y\in H$.
 Hence $$\mu=\frac{1}{2}\left(E_{x_0}+E_{-x_0}\right).$$ 
 
 2. Let us prove the theorem supposing that $X=\mathbb{R}^n$. Then $Y=\mathbb{R}^n$.
Denote by $K$ the minimal closed subgroup of $\mathbb{R}^n$ containing the support of
$\mu$ and by $L$ the character group of the group $K$. 
We can consider $\xi_1$  and $\xi_2$ as independent random variables 
with values in $K$. Then, by Lemma \ref{l3}, 
\begin{equation}\label{20.1}
\{y\in L:\hat\mu(y)=1\}=\{0\}.
 \end{equation}
It is well-known that any closed subgroup of the group $\mathbb{R}^n$ is 
topologically isomorphic to a group of the form $\mathbb{R}^p\times \mathbb{Z}^q$, 
where $p+q\le n$. Without loss of generality, we may assume that
$K=\mathbb{R}^p\times \mathbb{Z}^q$ and  
$L=\mathbb{R}^p\times \mathbb{T}^q$. 

Suppose that $p>0$. 
Fix $(s_1,\dots,s_p)\in \mathbb{R}^p\subset L$ and denote by $\pi$ the 
homomorphism
$\pi: \mathbb{R}\rightarrow \mathbb{R}^p\subset L$ of the form 
$$\pi(s)=(ss_1,\dots,ss_p), \quad s\in \mathbb{R}.
$$
Put $f(s)=\hat\mu(\pi(s))$,  $s\in \mathbb{R}$.
 Then $f(s)$ is a characteristic function on $\mathbb{R}$, and
$f(s)$ satisfies  equation (\ref{1n26.14}) because  $\hat\mu(y)$
satisfies equation (\ref{n26.14}). Since $\mu$ is a symmetric distribution,  
the function $f(s)$ is real-valued.
It follows from Lemma
\ref{nl1} and Theorem B that $f(s)=\cos as$ for some real number $a$. Since
$f(s)=\hat\mu(\pi(s))$, this  
 contradicts (\ref{20.1}). Hence $p=0$, i.e., $K=\mathbb{Z}^q$. The 
assertion of the theorem follows from the statement proved in item 1.

3. Let us prove the theorem in the general case, 
when $X=\mathbb{R}^n\times G$, where $n>0$, and $G\ne\{0\}$.
Let $\tau$ be the continuous homomorphism
$\tau: \mathbb{R}\rightarrow H$ such as in item 1. Denote by $(s_1,\dots,s_n,h)$, 
where $(s_1,\dots,s_n)\in \mathbb{R}^{n}$, $h\in H$, elements of the group 
$Y=\mathbb{R}^n\times H$ and by
$\kappa$
  the continuous homomorphism 
$\kappa: \mathbb{R}^{n+1}\rightarrow \mathbb{R}^n\times H$ of the form
$$
\kappa(s_1,\dots,s_n,s_{n+1})=(s_1,\dots,s_n,\tau(s_{n+1})), \quad
(s_1,\dots,s_n,s_{n+1})\in \mathbb{R}^{n+1}.
$$ 
Then the subgroup $\kappa(\mathbb{R}^{n+1})$ is dense in $\mathbb{R}^n\times H$.
Consider the restriction of equation (\ref{n26.14}) to the subgroup 
 $\kappa(\mathbb{R}^{n+1})$, i.e., assume that $y=\kappa(s_1,\dots,s_n,s_{n+1})$, 
 $(s_1,\dots,s_n,s_{n+1})\in \mathbb{R}^{n+1}$,
  in (\ref{n26.14}).
Put $$f(s_1,\dots,s_n,s_{n+1})=\hat\mu(\kappa(s_1,\dots,s_n,s_{n+1})), 
 \quad (s_1,\dots,s_n,s_{n+1})\in \mathbb{R}^{n+1}.$$
 Then $f(s_1,\dots,s_n,s_{n+1})$ is a characteristic function on $\mathbb{R}^{n+1}$
 satisfying the equation 
\begin{equation*} 
 f(2(s_1,\dots,s_n,s_{n+1}))=2f^2(s_1,\dots,s_n,s_{n+1})-1, \quad (s_1,\dots,s_n,s_{n+1})\in \mathbb{R}^{n+1}.
 \end{equation*}
Since $\mu$ is a symmetric distribution,  
the function $f(s_1,\dots,s_n,s_{n+1})$ is real-valued.
Taking into account Lemma \ref{nl1}, it follows from the statement proved in item 2, 
that there is $(t_1,\dots,t_n,t_{n+1})\in \mathbb{R}^{n+1}$ such that
$$
f(s_1,\dots,s_n,s_{n+1})=\cos(t_1s_1+\dots+t_ns_n+t_{n+1}s_{n+1}), \quad 
(s_1,\dots,s_n,s_{n+1})\in \mathbb{R}^{n+1}.
$$
Hence the function $f(s_1,\dots,s_n,s_{n+1})$ satisfies  d'Alembert's functional equation
on $\mathbb{R}^{n+1}$. This implies
that the function $\hat\mu(y)$ satisfies  d'Alembert's functional equation 
on the subgroup $\kappa(\mathbb{R}^{n+1})$.
 Since the subgroup $\kappa(\mathbb{R}^{n+1})$ is dense in $Y$, the function $\hat\mu(y)$ 
 satisfies  d'Alembert's functional equation on the group $Y$. By O'Connor 
 theorem,  applying to the function 
 $h(y)=\hat\mu(y)$, $y\in Y$, there is $x_0\in X$ such that 
 $\hat\mu(y)={\rm Re}\ (x_0, y)$, $y\in H$.
 Hence $\mu=\frac{1}{2}\left(E_{x_0}+E_{-x_0}\right)$. 
  
The statement (ii)$\Longrightarrow$(i) is obvious.
\end{proof}

\begin{corollary}\label{nco.1} Let $Y$ be a locally compact connected Abelian group. 
Let $h(y)$ be a real-valued characteristic function on $Y$
satisfying the equation 
\begin{equation}\label{n19.1}
h(2y)=2h^2(y)-1, \quad y\in Y.
 \end{equation}
Then there is $x_0\in X$, where $X$ is 
the character group of the group $Y$, such that $h(y)={\rm Re}\ (x_0, y)$, 
$y\in Y$.
\end{corollary}

We note that equation (\ref{n19.1}) follows from 
d'Alembert's functional equation (\ref{18.1}) if we put $u=v=y$ in (\ref{18.1}).

\section{Some Remarks}

\begin{remark}\label{re18.1}
To prove Theorems A and B
in \cite{Kl2023}, 
L.B.~Klebanov used the method of intensively monotone 
operators that he had developed earlier (see the monograph 
 by  A.V.~Kakosyan, L.B.~Klebanov, and J.A.~Melamed \cite{KaKlMe}). We note that 
the following Linnik's theorem\footnote{Formally speaking, this statement 
was not proven by Yu.V.~Linnik, but can be easily 
derived with the help of arguments which he 
used for the proof of Theorem 4.2.1 in
\cite{Li} (for the proof of Linnik's theorem see \cite{RR}).} 
implies Theorems A and B.
\begin{D}
Assume that a characteristic function $g(s)$ has moments of all orders and 
$g(s)$ is uniquely defined by them, in particular $g(s)$ can be
an analytical. Let $f(s)$ be a characteristic 
function such that $f(s_n)=g(s_n)$, where $s_n$ is a sequence such that
$s_n\rightarrow 0$ as $n\rightarrow\infty$. Then $f(s)=g(s)$ for all 
$s\in\mathbb{R}$.
\end{D}
1. By Lemma \ref{l1}, to prove Theorem A it suffices to show that 
if a real-valued characteristic function
$f(s)$ satisfies equation (\ref{1.5}), then $f(s)$ is represented in the 
form (\ref{1.3}) for some real $\sigma$. Take $s_0>0$ in such a way that 
$0< f(s)\le 1$ for all $0\le s\le s_0$. Find $\sigma$ 
from the equation
$$
f(s_0)=\frac{2}{e^{\sigma s_0}+{e^{-\sigma s_0}}}
$$
and put 
$$
g(s)=\frac{2}{e^{\sigma s}+{e^{-\sigma s}}}, \quad s\in \mathbb{R}.
$$
The function $g(s)$ is analytical in some neighborhood of the zero. 
 We have $f(s_0)=g(s_0)$. Since both the 
functions $f(s)$ and $g(s)$ satisfy 
equation (\ref{1.5}), we find sequentially from (\ref{1.5}) the equalities
\begin{equation}\label{23.06.1}
f\left(\frac{s_0}{2^n}\right)=g\left(\frac{s_0}{2^n}\right), \quad n=1, 2, \dots. 
 \end{equation}

By Linnik's theorem, $f(s)=g(s)$ for all $s\in\mathbb{R}$. Moreover, 
this argument shows that in order to assert that $f(s)=g(s)$ for all $s\in\mathbb{R}$,
 it suffices to assume that the function $f(s)$ satisfies equation 
 (\ref{1.5}) in some neighborhood of the zero, but not on the entire real line.
 
2. By Lemma \ref{nl1}, to prove Theorem B it suffices to show that 
if a real-valued characteristic function
$f(s)$ satisfies equation (\ref{1n26.14}), then $f(s)=\cos as$  
 for some real $a$. Take $s_0>0$ in such a way that 
$0< f(s)\le 1$ for all $0\le s\le s_0$. Put
$$
a=\frac{\arccos f(s_0)}{s_0} 
$$
and
$$
g(s)=\cos as, \quad s\in \mathbb{R}.
$$
The function $g(s)$ is analytical in some neighborhood of the zero. We have
 $f(s_0)=g(s_0)$ and $g(s)>0$ for all $0\le s\le s_0$.

We continue our reasoning in the same way as in item 1, using 
equation (\ref{1n26.14}) to obtain (\ref{23.06.1}) instead of equation (\ref{1.5}).
\end{remark}

\begin{remark}\label{re18.2}
O'Connor's theorem  was generalized in \cite{L}
 for 
arbitrary Abelian groups. Therefore, 
we can formulate the following natural question: is Corollary \ref{nco.1} true 
if we omit the condition  that the group $Y$ is connected.
If the answer is positive, then it will follow that Theorem \ref{nth1}
is valid for an arbitrary locally compact Abelian group.
\end{remark}

\medskip

\noindent B. Verkin Institute for Low Temperature Physics and Engineering of the National\\  Academy of Sciences of Ukraine, 47, Nauky ave, Kharkiv, 61103, Ukraine\\

\noindent E-mail: feldman@ilt.kharkov.ua

\end{document}